\documentclass[10pt]{article}
\font\smallit=cmti10

\usepackage{amssymb,latexsym,amsmath,epsfig,amsthm} 
\usepackage{enumitem}

\usepackage[linesnumbered,ruled,noline]{algorithm2e}
\SetKwIF{If}{ElseIf}{Else}{if}{}{else if}{else}{end}%
\SetKwFor{For}{for}{}{end}
\SetKwFor{While}{while}{}{end}
\SetKwProg{Fn}{function}{}{end}
\SetKwComment{Comment}{ /\!/\,}{}
\SetCommentSty{itshape}

\makeatletter

\renewcommand\section{\@startsection {section}{1}{\z@}
{-30pt \@plus -1ex \@minus -.2ex}
{2.3ex \@plus.2ex}
{\normalfont\normalsize\bfseries\boldmath}}

\renewcommand\subsection{\@startsection{subsection}{2}{\z@}
{-3.25ex\@plus -1ex \@minus -.2ex}
{1.5ex \@plus .2ex}
{\normalfont\normalsize\bfseries\boldmath}}

\renewcommand{\@seccntformat}[1]{\csname the#1\endcsname. }

\makeatother

\newtheorem{theorem}{Theorem}
\newtheorem{lemma}{Lemma}
\newtheorem{proposition}{Proposition}

\theoremstyle{definition}

\newtheorem{remark}{Remark}

\newcommand{\coeff}{\mathrm{coeff}}
\newcommand{\ee}{\mathrm{e}}

\renewcommand{\AA}{\mathtt{A}}
\newcommand{\BB}{\mathtt{B}}

\newcommand{\QQ}{\mathbb{Q}}

\newcommand{\denom}{\mathrm{denom}}
\newcommand{\lcm}{\mathrm{lcm}}

\begin{document}


\begin{center}
\uppercase{\bf Smallest common denominators for the homogeneous 
components of the Baker--Campbell--Hausdorff series}
\vskip 20pt
{\bf Harald Hofst\"atter}\\ 
{\smallit Reitschachersiedlung 4/6, 7100 Neusiedl am See, Austria}\\
{\tt hofi@harald-hofstaetter.at}\\ 
\end{center}
\vskip 20pt



\centerline{\bf Abstract}

\noindent
In a recent paper the author derived a formula for calculating
common denominators for the homogeneous components of the
Baker--Campbell-Hausdorff (BCH) series. 
In the present work it is proved that this formula actually yields the  smallest
such common denominators. In an appendix a new efficient algorithm for
 computing  coefficients of the BCH series is presented, which is 
based on these  common denominators, and requires only integer arithmetic rather than less efficient rational arithmetic.

\pagestyle{myheadings}
\thispagestyle{empty}
\baselineskip=12.875pt
\vskip 30pt

\section{Introduction}
We continue the 
investigations from our recent paper \cite{HHdenom} 
on the
Baker--Campbell--Hausdorff (BCH) series, which is  formally defined as the element 
\begin{equation*}
H = \log(\ee^{\AA}\ee^{\BB})
= \sum_{k=1}^\infty\frac{(-1)^{k+1}}{k}\big(\ee^{\AA}\ee^{\BB}-1\big)^k
= \sum_{k=1}^\infty\frac{(-1)^{k+1}}{k}\bigg(\sum_{i+j>0}\frac{1}{i!j!}\AA^i\BB^j\bigg)^k 
\end{equation*}
in the ring $\QQ\langle\langle\AA,\BB\rangle\rangle$  of formal power series 
in the non-commuting variables $\AA$ and $\BB$ with rational coefficients.
The BCH series can be written as a sum $H=\sum_{n=1}^{\infty}H_n$
of homogeneous components 
\begin{equation*}
H_n= \sum_{w\in\{\AA,\BB\}^n} h_w w,\quad n=1,2,\dots,
\end{equation*}
where $\{\AA,\BB\}^n$
denotes the finite set of all words $w=w_1\cdots w_n$ ($w_i\in\{\AA,\BB\}$) of length (degree) $n$ over the 
alphabet $\{\AA, \BB\}$, and
$h_w=\coeff(w,H)$ denotes the coefficient of such a word in 
$H$.
The main result of the present paper is the following strengthening
of  Theorem~1 of  \cite{HHdenom} on common denominators for these homogeneous components $H_n$.
\begin{theorem}\label{Thm:MainTheorem}
For $n\geq 1$ and prime $p\geq 2$, let
\begin{equation}\label{eq:l_n_p}
l(n,p)=\max\{t:\  p^t\leq s_p(n)\} = \lfloor\log_p(s_p(n))\rfloor,
\end{equation}
where $s_p(n)=\alpha_0+\alpha_1+\ldots+\alpha_r$ is  the sum of
the digits in the $p$-adic expansion  $n=\alpha_0+\alpha_1p+\ldots+\alpha_rp^r$, $0\leq\alpha_i\leq p-1$.
Define 
\begin{equation}\label{eq:d_n}
d_n = \prod_{p\ \mathrm{prime}, \ p<n}
p^{l(n,p)},\quad n=1,2,\dots.
\end{equation}
Then $n!\,d_n$ is the smallest common denominator
 of all coefficients of words of length $n$ in 
 the Baker--Campbell--Hausdorff series $H=\log(\ee^{\AA}\ee^{\BB})$, 
 or, equivalently,\footnote{Here, $\lcm\,\mathcal{M}$ denotes the {\em least common multiple} of the elements of the finite set $\mathcal{M}\subset\mathbb{Z}$, and $\denom(r)$ 
for $r\in\mathbb{Q}$ is defined as the smallest
positive integer $d$ such that $r\cdot d\in\mathbb{Z}$.
In particular, $\denom(0)=1$.}
\begin{equation}\label{eq:lcm_smallest_denom}
\mathrm{lcm}\big\{\mathrm{denom}(\coeff(w,H)):\ w\in\{\AA,\BB\}^n
\big\}=n!\,d_n.
\end{equation}
\end{theorem}
 The weaker statement that $n!\,d_n$ is a (not necessarily the
smallest) common denominator, or, equivalently, 
\begin{equation}\label{eq:lcm_denom}
\mathrm{lcm}\big\{\mathrm{denom}(\coeff(w,H)):\ w\in\{\AA,\BB\}^n
\big\}\ | \ n!\,d_n
\end{equation}
was proved in \cite{HHdenom}, where it was also proved that
\begin{equation}
n!\,d_n = \lcm\{k\, j_1!\cdots j_k!:\ j_i\geq 1, \ j_1+\ldots+j_k=n,\ k=1,\dots,n\}.
\end{equation}

The first few values of $d_n$ are
\begin{equation*}
d_n = 1, 1, 2, 1, 6, 2, 6, 3, 10, 2, 6, 2, 210, 30, 12, 3, 30, 10, 210, 42, 330, 30, 60, 30, 546, \ldots, 
\end{equation*}
see 
\cite[A338025]{SloaneOEIS}. 

\begin{algorithm}
\caption{Construction of a specific word $w(n,p)$  that satisfies (\ref{eq:vp_wnp})}
\label{alg:witness}
\DontPrintSemicolon
\KwIn{$n\geq 1$ and $p\geq 2$ prime} 
\KwOut{$(q_1,\dots, q_m)$ such that $w(n,p)=\AA^{q_1}\BB^{q_2}\AA^{q_3}\cdots (\AA\lor\BB)^{q_m}$
satisfies (\ref{eq:vp_wnp}) }
Determine $r$, $\alpha_0,\dots,\alpha_r$ such that 
$n=\alpha_0+\alpha_1p+\ldots+\alpha_rp^r$, $0\leq \alpha_i\leq p-1$\;
$s:=\alpha_0+\ldots+\alpha_r$ \Comment*[r]{\ $s=s_p(n)$}
$l:=
\lfloor\log_p(s)\rfloor$ 
\Comment*[r]{\ $l=l(n,p)$ as defined in (\ref{eq:l_n_p})}
\uIf{$l=0$} {
	\uIf(\Comment*[f]{see the paragraph after Remark~2}){$n<p$} {
		\uIf{$n=1$} {
			{\bf return} $(1)$
		}
		\uElseIf{$n=2$ {\bf or} $n$ \rm odd} {
			{\bf return} $(n-1, 1)$
		}
		\Else {
			{\bf return} $(n-2, 2)$
		}
	}
	\Else(\Comment*[f]{define $k$ as in  Lemma~\ref{lemma:1}}){
		$k:=p^{r-1}(p-1)$\;
		{\bf return} $(n-k,k)$
	}
}
\uElseIf(\Comment*[f]{define $k$ with $1\leq k\leq n$ and $p-1\mid k$ as in  Lemma~\ref{lemma:2}}){$l=1$} {
	$h:=p-1$\; 
	$i:=0$\;
	\While{$h> 0$} {
		$\beta_i:=\min\{h,\alpha_i\}$\;
        $h:=h-\beta_i$\;
        $i:=i+1$        		
	}
	$k:=\beta_0+\beta_1p+\ldots+\beta_{i-1}p^{i-1}$\;
	{\bf return} $(n-k,k)$
}
\Else(\Comment*[f]{define  $(q_1,\dots,q_m)$ as in 
 Section~\ref{SubSec:l_gt_1}}) {
	\uIf{$p=2$ {\bf or} $n$ \rm odd} {
		$m:=p^l$ \Comment*[r]{$a\not\equiv 0\ (\mathrm{mod}\ p )$ in (\ref{eq:cauv0}) for $p=2$ or $n$ odd}
	}
	\Else {
		$m:=p^l+1$	\Comment*[r]{
		$a\not\equiv 0\ (\mathrm{mod}\ p )$ in (\ref{eq:cauv1}) for $p\neq 2$}
	}
	$(b_1,\ldots,b_{s}):=(\underbrace{p^0,\dots,p^0}_{\alpha_0},
	 	 \underbrace{p^1,\dots,p^1}_{\alpha_1},
	 	 \dots,
	 	 \underbrace{p^r,\dots,p^r}_{\alpha_r})$\;
	{\bf return} 
		$(b_{m}+b_{m+1}+\ldots+ b_{s}, b_{m-1},b_{m-2},\dots, b_1)$
}
\end{algorithm}
For the proof of Theorem~\ref{Thm:MainTheorem} 
we  explicitly
construct, for each degree $n$ and each prime $p\geq 2$,
a specific word $w(n,p)\in\{\AA,\BB\}^n$, 
for which 
\begin{equation}\label{eq:vp_wnp}
v_p(\denom(\coeff(w(n,p), H)))
=v_p(n!)+l(n,p)=v_p(n!\,d_n),
\end{equation}
where $v_p(k)$ denotes
the exponent of the
highest power of $p$ that divides  $k$.
Since 
$v_p(\denom(\coeff(w, H)))\leq v_p(n!\,d_n)$,
$w\in\{\AA,\BB\}^n$
 by (\ref{eq:lcm_denom}), 
this  implies
\begin{equation}\label{eq:max_vp}
\max_{w\in\{\AA,\BB\}^n}v_p(\denom(\coeff(w, H)))=v_p(n!\,d_n),\quad
p\geq 2\ \mbox{prime},
\end{equation}
and thus
(\ref{eq:lcm_smallest_denom})
by unique factorization.
An overview of the construction
of $w(n,p)$ is provided by 
Algorithm~\ref{alg:witness}.\footnote{Algorithm~\ref{alg:witness} serves as a guide 
for the proof of Theorem~\ref{Thm:MainTheorem} in Section~\ref{Sec:proof}. It 
is 
detailed enough for a direct implementation in a computer algebra system. In particular, it yields a  definite $w(n,p)$ for which the discussion
in Section~\ref{Sec:proof} would allow several possibilities.}
The details of the proof of Theorem~\ref{Thm:MainTheorem}
are given  in Section~\ref{Sec:proof}.
Although these details 
may seem rather long and technical, the proof
still has a certain appeal as it uses classical results  of
Lucas, von Staudt and Clausen, Hermite and Bachmann,  Glaisher, etc.~on divisibility and congruence properties of binomial coefficients, Bernoulli numbers and Stirling numbers of the second kind. 

In Section~\ref{Sec:numerical_illustrations} we 
illustrate some of our results of Section~\ref{Sec:proof}
by  explicit computations.
These computations are based on
a new efficient algorithm for the computation of coefficients of the BCH series, which is described in detail in the appendix, and
which uses the common 
denominators $n!\,d_n$ 
in an essential way. This brings us back to our original motivation for the investigation of such common denominators in \cite{HHdenom}. The algorithm can
be implemented in a straightforward way, it performs all computations in integer arithmetic, and, unlike the algorithm described in \cite{VanBruntVisser},
it requires no symbolic manipulation software.

\section{Proof of Theorem~\ref{Thm:MainTheorem}}\label{Sec:proof}
We base our investigations on  explicit formulas 
due to Goldberg \cite{G} for the coefficients
of the BCH series, which are given by the following two propositions.
\begin{proposition}
Let $w=\AA w_2\cdots w_n \in \{\AA, \BB\}^n$ be a word of degree $n\geq 1$ starting
with the letter $\AA$. 
Let $q_1,\dots,q_m\geq 1$ with $q_1+\ldots+q_m=n$ such that 
\begin{equation}\label{eq:wAq}
w=\AA^{q_1}\BB^{q_2}\AA^{q_3}\cdots (\AA\lor\BB)^{q_m}, 
\end{equation}
where 
$\AA\lor\BB$ denotes $\AA$ if $m$ is odd and $\BB$ if
$m$ is even.
Then
\begin{equation*}
\coeff(w, \log(\ee^{\AA}\ee^{\BB})) = c(q_1,\dots,q_m)
= \frac{(-1)^n}{q_1!\cdots q_m!}\widetilde{c}(q_1,\dots,q_m)
\end{equation*}
with
\begin{align}
\widetilde{c}(q_1,&\dots, q_m) = \nonumber \\ \label{eq:coeff_expl}
&
\sum_{k=0}^{\widetilde{m}}
\!
\sum_{\substack{1\leq j_1\leq q_1,\\ \dots,\\1\leq j_m\leq q_m}}
\!\!\!\!\!
(-1)^{j_1+\ldots+j_m-k}
\genfrac(){0pt}{}{\widetilde{m}}{k}\frac{j_1!\cdots j_m!}{j_1+\ldots+j_m-k}
S(q_1,j_1)\cdots S(q_m,j_m),
\end{align}
where
\begin{equation*}
\widetilde{m} = \left\lfloor\frac{m-1}{2}\right\rfloor,
\end{equation*}
and the  $S(q,j)$ denote Stirling numbers of the second kind defined by
\begin{equation}\label{eq:stirling}
S(q,j) = \frac{1}{j!}\sum_{i=0}^j(-1)^{j-i}\genfrac(){0pt}{}{j}{i}i^q.
\end{equation}
\end{proposition}
\begin{proof}
The proposition 
 follows by simple manipulations from Theorem~1 of
\cite{G}, which states that (using denotations from \cite{G})
\begin{equation*}
\coeff(w, \log(\ee^{\AA}\ee^{\BB})) =  \int_0^1 
t^{m'}(t-1)^{m''}
G_{q_1}(t)\cdots G_{q_m}(t)\,\mathrm{d}t,
\end{equation*}
where $m'=m-\widetilde{m}-1$, $m''=\widetilde{m}$, and
$G_q(t) = \sum_{j=1}^q (-1)^{q-j}\alpha_q^{(j)}t^{j-1}$, 
$\alpha_q^{(j)}=\frac{j!}{q!}S(q,j)$, 
such that
\begin{equation*}
G_q(t) = \frac{(-1)^q}{q!}\sum_{j=1}^{q}(-1)^j j!S(q,j)t^{j-1}.
\end{equation*}
\end{proof}
\begin{remark}\label{rem:1} For words $w=\BB w_2\cdots w_n$ starting with 
$\BB$ we have $\coeff(w, \log(\ee^{\AA}\ee^{\BB})) = (-1)^{n+1}c(q_1,\dots,q_m)$ with $q_i$ analogously defined as in (\ref{eq:wAq}) and $c(q_1,\dots,q_m)$ again given by (\ref{eq:coeff_expl}), see \cite{G}.
Also note that $c(q_1,\dots,q_m)$ is invariant under permutations of the $q_i$.
If follows that if $n$ is even and $m$ is odd, then $c(q_1,\dots,q_m)=0$.
\end{remark}
\begin{proposition}\label{prop:m2_expl}
For words of the form $w=\AA^{n-k}\BB^k\in\{\AA,\BB\}^n$, $1\leq k\leq n-1$, 
$n\geq 2$
we have
\begin{equation}\label{eq:m2_expl}
\coeff(w,\log(\ee^\AA\ee^\BB))=c(n-k,k) = 
\frac{(-1)^{n+k}}{n!}\genfrac(){0pt}{}{n}{k}\sum_{j=1}^k\genfrac(){0pt}{}{k}{j}B_{n-j},
\end{equation}
where $B_n$ denote the Bernoulli numbers ($B_1=-\tfrac{1}{2}$).
\end{proposition}
\begin{proof}
See \cite[Theorem~3]{G}.
\end{proof}

Depending on the integer $n\geq 1$ and the prime $p\geq 2$ we are now going to construct specific integers
$q_1,\dots, q_m\geq 1$  which satisfy 
 $q_1+\ldots+q_m=n$ and
\begin{equation}\label{eq:vp_den_c}
v_p(\denom(c(q_1,\dots,q_m)))=v_p(n!\,d_n)
=v_p(n!)+l(n,p).
\end{equation}
Then 
$\AA^{q_1}\BB^{q_2}\cdots(\AA\lor\BB)^{q_m}$
defined with these integers  is a suitable word $w(p,n)\in\{\AA,\BB\}^n$
for (\ref{eq:vp_wnp}), which 
suffices for 
the proof of Theorem~\ref{Thm:MainTheorem}.

Here and in the following 
$v_p(k)$ 
denotes the {\em $p$-adic valuation} of $k$, i.e., the exponent of the
highest power of the prime  $p$ that divides the integer $k$. By convention, $v_p(0)=\infty$. More generally, for rationals $u/v$ with  
$u, v\in\mathbb{Z}$, $v\neq 0$ (not necessarily in lowest terms),  $v_p(u/v)=v_p(u)-v_p(v).$ For the computation of $v_p$ for factorials we will
use {\em Legendre's formula}
\begin{equation*}
v_p(k!)=\frac{k-s_p(k)}{p-1}, \qquad k=1,2,\dots,
\end{equation*}
where, as in (\ref{eq:l_n_p}), $s_p(k)$ is the sum of the digits in the 
$p$-adic expansion of $k$, cf.~\cite{Mihet}.

\subsection{The case $l(n,p)\leq 1$}\label{SubSec:l_leq_1}
We  first assume  $l(n,p)\leq 1$, i.e., $s_p(n)<p^2$,  and deal with  the case 
$l(n,p)\geq 2$ afterwards.
We try to find an integer $k$ depending on
$n$, $p$,
such that (\ref{eq:vp_den_c}) holds with $m=2$,  $q_1=n-k$, $q_2=n$.
For this we first examine the sum in 
(\ref{eq:m2_expl}). 
\begin{proposition}\label{prop:sum_m_2_mod_p}
Let $p\geq 2$ prime, $n\geq 2$, and $1\leq k\leq n-1$. 
Let   $r$ be the unique integer such that $r\equiv n\ (\mathrm{mod}\ p-1)$ and $0\leq r\leq p-2$, and let $\widetilde{k}$ be the unique integer such that $\widetilde{k}\equiv k\ (\mathrm{mod}\ p-1)$ and $1\leq \widetilde{k} \leq p-1$.\footnote{Note that here $\widetilde{k}\geq 1$ is required, so that in particular $\widetilde{k} = p-1$ if $p-1\mid k$.}
Then there exist an integer $a$ 
(unique modulo $p$)
and a rational number $U$, $v_p(U)\geq 0$, 
such that 
\begin{equation}\label{eq:sum_binom_bernoulli}
(-1)^{n+k}(n-k)!k!c(n-k,k)=
\sum_{j=1}^{k}\genfrac(){0pt}{}{k}{j}B_{n-j} =-\frac{a}{p}+U
\end{equation}
and
\begin{equation}\label{eq:sum_binom_bernoulli_2}
a\equiv \left\{\begin{array}{rl}
\genfrac(){0pt}{}{\widetilde{k}}{r}
\   (\mathrm{mod}\ p), & \mbox{if $r\geq 1$}, \\
1\ \ (\mathrm{mod}\ p), & \mbox{if $r=0$, $p-1\mid k$},\\
 0\ \ (\mathrm{mod}\ p), & \mbox{if $r=0$, $p-1\nmid k$}.
\end{array}\right.
\end{equation}
\end{proposition}
\begin{proof}
The Bernoulli numbers can be written in the form 
$$
B_n = \left\{\begin{array}{cl}
-\frac{1}{p} + U_n,& \mbox{if $p-1\mid n$},\\
U_n, &\mbox{if $p-1\nmid n$}
\end{array}\right.
\quad\mbox{for some}  \ U_n\in\mathbb{Q},  \ v_p(U_n)\geq 0, \quad n=1,2,\dots,
$$
which is an easy consequence of the  von~Staudt--Clausen theorem,
cf.~\cite{Carlitz}.
This implies (\ref{eq:sum_binom_bernoulli}) with $U=\sum_{j=1}^{k}\genfrac(){0pt}{}{k}{j}U_{n-j}$ and
\begin{equation}\label{eq:s_sum}
a=\sum_{\substack{1\leq j\leq k,\\ j\equiv r\, (\mathrm{mod}\ p-1)}}\genfrac(){0pt}{}{k}{j},
\end{equation}
which for $r\geq 1$ is $\equiv \genfrac(){0pt}{}{\widetilde{k}}{r}
\   (\mathrm{mod}\ p)$, which is a result due to Glaisher, cf.~\cite[Eq.~(1)]{Mattarei}.
In the case $r=0$  the congruence (\ref{eq:sum_binom_bernoulli_2}) follows 
from (\ref{eq:s_sum}) by an application of
$$
\sum_{\substack{1\leq j\leq k-1,\\ j\equiv 0\, (\mathrm{mod}\ p-1)}}
\genfrac(){0pt}{}{k}{j}
\equiv 0\ (\mathrm{mod}\ p),
$$ 
 a result of Hermite and Bachmann, cf.~\cite[Corollary~1]{MacMillanSondow}. (Note that here $j\leq k-1$ instead of 
 $\leq k$ as before.)
\end{proof}

\begin{remark}\label{rem:sum_m_2_mod_p} If $p-1\mid k$, then we have
\begin{equation*}
\genfrac(){0pt}{}{\widetilde{k}}{r} =
\genfrac(){0pt}{}{p-1}{r} = \frac{(p-1)\cdots(p-r)}{1\cdots r}
\equiv(-1)^r
\equiv(-1)^n
\ \ (\mathrm{mod}\ p)
\end{equation*}
in (\ref{eq:sum_binom_bernoulli_2}), where for the last congruence
we have  used $n=q(p-1)+r$ with $p-1$ even if $p\geq 3$  (and  $-1\equiv+1$ if $p=2$).
\end{remark}
If $n<p$, then  it follows easily from 
 Proposition~\ref{prop:sum_m_2_mod_p} that
$a\equiv 0\ (\mathrm{mod}\ p)$ in (\ref{eq:sum_binom_bernoulli}), and thus
$v_p\left(\sum_{j=1}^{k}\genfrac(){0pt}{}{k}{j}B_{n-j}\right)\geq 0$ for all $k=1,\ldots,n-1$.
Because 
$v_p\left(\genfrac(){0pt}{}{n}{k}\right)=0$, 
$v_p(n!)=0$,  and $l(n,p)=0$ for $n<p$, this implies (\ref{eq:vp_den_c}) for
$n<p$ and  all  $q_1=n-k$, $q_2=k$, $k=1,\ldots,n-1$.\footnote{For 
even $n\geq 4$ we have
$c(n-k,k)=0$ for $k=1$ and $k=n-1$, 
so that
in this case (\ref{eq:vp_den_c}) holds by the convention $\denom(0)=1$.
 If we want to avoid this convention, we can 
consider $k=2$ instead, since $c(n-2,2)\neq 0$ for $n\geq 3$.}

In the following we assume $n\geq p$. 

It follows from Proposition~\ref{prop:sum_m_2_mod_p} and Remark~\ref{rem:sum_m_2_mod_p} that if $p-1\mid k$, then
$a\not\equiv 0\ (\mathrm{mod}\ p)$ in (\ref{eq:sum_binom_bernoulli}), and
thus $v_p\left(\sum_{j=1}^{k}\genfrac(){0pt}{}{k}{j}B_{n-j}\right)=-1$ 
in this case. Therefore and according to (\ref{eq:m2_expl}),
it remains to find an integer $k$ with $1\leq k\leq n-1$  and $p-1\mid k$, for which
we have
$v_p\left(\genfrac(){0pt}{}{n}{k}\right)=1$ in the case 
$l(n,p)=0$,
or 
$v_p\left(\genfrac(){0pt}{}{n}{k}\right)=0$ in the case 
 $l(n,p)=1$. Such a $k$ then ensures that (\ref{eq:vp_den_c}) holds with 
 $m=2$, $q_1=n-k$, $q_2=k$. 
 The following two lemmas show respectively that 
a suitable such $k$ can indeed be chosen in both cases.

\begin{lemma}\label{lemma:1}
Let  $p\geq 2$ prime and  let $n\geq p$ be an integer with  $p$-adic expansion
$$n=\alpha_rp^r+\alpha_{r-1}p^{r-1}+\ldots+\alpha_0,\quad
r\geq 1, \ \alpha_r\geq 1, \ 0\leq \alpha_i\leq p-1,
$$
where $\alpha_{r-1}<p-1$ (which is certainly the case if $s_p(n)<p$).
Then
\begin{equation*}
v_p\left(\genfrac(){0pt}{}{n}{p^{r-1}(p-1)}\right) 
= 1. 
\end{equation*}
\end{lemma}
\begin{proof}
Let $k=p^{r-1}(p-1)$. Then $n-k=(\alpha_r-1)p^r+(\alpha_{r-1}+1)p^{r-1}
+\alpha_{r-2}p^{r-2}+\ldots+\alpha_0$, and thus
\begin{align*}
v_p&\left(\genfrac(){0pt}{}{n}{k}\right)
=\frac{1}{p-1}\big(s_p(k)+s_p(n-k)-s_p(n)\big)\\
&=\frac{1}{p-1}\big((p-1)+(\alpha_r-1)+(\alpha_{r-1}+1)+\alpha_{r-2}+\ldots+\alpha_0
-(\alpha_r+\ldots+\alpha_0)\big)\\
&=1.
\end{align*}
\end{proof}
\begin{lemma}\label{lemma:2}
Let $p\geq 2$ prime and  let $n\geq 1$ be an integer with  $p$-adic expansion
$$n=\alpha_0+\alpha_1 p+\ldots+\alpha_r p^r,
\quad
0\leq \alpha_i\leq p-1.
$$
If  an integer $k\geq 0$ has a $p$-adic expansion of the special form
\begin{equation}\label{eq:k_beta_leq_alpha}
k=\beta_0+\beta_1 p+\ldots+\beta_r p^r,
\quad
0\leq\beta_i\leq \alpha_i,
\end{equation}
then
\begin{equation*}
v_p\left(\genfrac(){0pt}{}{n}{k}\right) 
= 0.
\end{equation*}
Furthermore, if $s_p(n)=\alpha_0+\ldots+\alpha_r\geq p$, then
there exists an integer  $k$ of the form (\ref{eq:k_beta_leq_alpha})
which satisfies $1\leq k\leq n-1$ and $p-1\mid k$.
\end{lemma}
\begin{proof}From the given conditions it follows that
$n-k$ has the $p$-adic expansion
 $n-k=\sum_{i=0}^{r}(\alpha_i-\beta_i)p^i$, $1\leq \alpha_i-\beta_i\leq p-1$. Thus, $s_p(n-k)=s_p(n)-s_p(k)$ and 
$v_p\left(\genfrac(){0pt}{}{n}{k}\right) =\frac{1}{p-1}\big(
s_p(k)-s_p(n-k)+s_p(n)\big)=0$.

Now assume $s_p(n)=\alpha_0+\ldots+\alpha_r\geq p$.
Choose integers $\beta_0,\dots,\beta_r$
such that $0\leq\beta_i\leq \alpha_i$ and
$\beta_0+\ldots+\beta_r=p-1$ and define $k$ according to (\ref{eq:k_beta_leq_alpha}). Note that this is possible because
$p-1 <\alpha_0+\ldots+\alpha_r$. Then $1\leq k\leq n-1$ and
$k=\beta_0+\beta_1 p+\ldots+\beta_r p^r\equiv \beta_0+\ldots+\beta_r
\equiv 0 \ (\mathrm{mod} \ p-1)$.
\end{proof}
\subsection{\label{SubSec:l_gt_1} The case 
$l(n,p)\geq 1$}
Assuming $l(n,p)\geq 1$, i.e., $s_p(n)\geq p$, we construct integers $q_1,\dots,q_m\geq 1$ 
depending on $n$, $p$, and $l=1,\dots, l(n,q)$, which satisfy 
\begin{equation}\label{eq:vp_den_c_l}
v_p(\denom(c(q_1,\dots,c_m))) = v_p(n!)+l.
\end{equation}
This is more general than strictly  necessary, since for 
the proof of Theorem~\ref{Thm:MainTheorem}
it would be sufficient to consider only
 the case $l=l(n,p)\geq 2$. Notice that
the case $l(n,p)=1$ has been dealt with previously.

We set
$$m=p^l \ \ \mbox{or} \ \ m=p^l+1.$$
Consider the $p$-adic expansion
\begin{equation}\label{eq:rep_n_sum_powers}
n = \alpha_0+\alpha_1 p+\ldots+\alpha_r p^r
=\underbrace{p^0+\ldots+p^0}_{\alpha_0}
+\underbrace{p^1+\ldots+p^1}_{\alpha_1}+\ldots
+\underbrace{p^r+\ldots+p^r}_{\alpha_r}.
\end{equation}
From this collection of $\alpha_0+\dots+\alpha_r=s_p(n)$
 powers of $p$ select $m-1$ powers 
$p^{e_2},\dots,p^{e_m}$ 
 and define with them
$$q_2 = p^{e_2},\ \dots,\ q_m=p^{e_m}.$$
Here, a specific
exponent $e$ occurs among the exponents $e_2,\dots,e_m$
at most $\alpha_e$ times. The remaining powers are 
collected in $q_1$ such that
%
$$q_1 = n-\sum_{i=2}^mq_i,$$ 
which ensures that  the  $q_i$ satisfy
$$\sum_{i=1}^{m}q_i=n\ \ \ \mbox{and}\ \ \ \sum_{i=1}^{m}s_p(q_i) = s_p(n)=\alpha_0+\dots+\alpha_r.$$
For this construction to be feasible, we 
have to assume that $q_1>0$.
For $m=p^l$ this is automatically the case, but 
for $m=p^l+1$ we have $q_1=0$ if and only if $s_p(n)=m-1=p^l$, which for $p\neq 2$ 
can only be the case if $n$ is odd, because for $p\neq 2$ and $n$ even we have $s_p(n)$ even but $p^l$ odd. 

If $q_1>0$ can be achieved, then even
$q_1\geq p$
 is possible for suitable chosen $q_i$. Indeed, because
$s_p(n)\geq p>\alpha_0$ at least one of the powers in the collection
of powers in 
(\ref{eq:rep_n_sum_powers}) has exponent $\geq 1$. Thus, the above
construction can be carried out in such a way that this power $\geq p$ 
ends up as one of the remaining powers whose sum is $q_1$.

The following proposition shows that for each $n\geq 1$, prime $p\geq 2$,
and $l\geq 1$,
in at least one of the cases\footnote{More specifically, case
$m=p^l$ for $p=2$ or $p\neq 2$ and $n$ odd, and 
case $m=p^l+1$ for $p\neq 2$ and $n$ even.
} $m=p^l$ or $m=p^l+1$ the above construction
 leads to integers $q_1,\dots,q_m$ which satisfy
$$\widetilde{c}(q_1,\dots,q_m) = \frac{a}{p^l}
+U,\quad a\in\mathbb{Z}, \ a\not\equiv 0\ (\mathrm{mod}\ p),\quad U\in\mathbb{Q},\ v_p(U)>-l,$$
where
$$c(q_1,\dots,q_m)=\frac{(-1)^n}{q_1!\cdots q_m!}\widetilde{c}(q_1,\dots,q_m).
$$
Thus,  
$v_p(\denom(\widetilde{c}(q_1,\dots,q_m))=l$
for these $q_i$, which  together with 
\begin{equation*}
v_p(q_1!\cdots q_m!) =
\frac{1}{p-1}\left(\sum_{i=1}^{m}q_i-\sum_{i=1}^{m}s_p(q_i)\right)
 = \frac{1}{p-1}\big(n-s_p(n)\big)
=v_p(n!)
\end{equation*}
implies (\ref{eq:vp_den_c_l}).
\begin{proposition}\label{prop:main_prop}
Let $n\geq 1$ and $p\geq 2$ prime.
Suppose 
$1\leq l\leq l(n,p)$ 
and let $m=p^l$ or $m=p^l+1$.
If $m=q^l+1$, we additionally assume that 
$s_p(n)\neq p^l$ (which is automatically the case if $n$ is even).
Depending on $n$ and  $p$ 
let $q_1,\dots,q_m$ with $q_1\geq p$ be defined as described above.
Furthermore, in the special case $p=2$, $l=1$, $m=p^l=2$, and $n$ odd, we assume $q_1=n-1$, $q_2=1$.
 Then 
there exists an integer $a$ (unique modulo $p$)  and a rational number $U,\ v_p(U)>-l$, such that
 $\widetilde{c}(q_1,\dots,q_m)=(-1)^nq_1!\cdots q_m!c(q_1,\dots,q_m)$ can 
be written as
\begin{equation}\label{eq:cauv}
\widetilde{c}(q_1,\dots,q_m) = \frac{a}{p^l}
+U.
\end{equation}
If $m=p^l$, then $a$ satisfies\footnote{If $p\neq 2$, then $m=p^l$ is odd. Thus, if $p\neq 2$ and $n$ is even, then not only $a\equiv 0\ (\mathrm{mod}\ p)$ but actually $c(q_1,\dots,q_m)=0$ by Remark~\ref{rem:1}.}
\begin{equation}\label{eq:cauv0}
a \equiv \left\{ 
\begin{array}{rl}
1 \ (\mathrm{mod}\ p), & \mbox{if $p=2$},\\
0 \ (\mathrm{mod}\ p), & \mbox{if $p\neq 2$ and $n$ even},\\
2\left(\frac{p-1}{2}\right)^{n}
\not\equiv 0
\ (\mathrm{mod}\ p),
&\mbox{if $p\neq 2$ and $n$ odd}.
\end{array}\right.
\end{equation}
If $m=p^l+1$, then  $a$ satisfies\footnote{If $p=2$ and $l\geq 1$, then $m=p^l+1$ is odd. Thus, similarly as before, if $p=2$, $l\geq 1$ and $n$ is even, then actually $c(q_1,\dots,q_m)=0$.}
\begin{equation}\label{eq:cauv1}
a \equiv \left\{ 
\begin{array}{rl}
1 \ (\mathrm{mod}\ p), & \mbox{if $p=2$, $l=1$, $n$ odd},\\
0 \ (\mathrm{mod}\ p), & \mbox{if $p=2$, $l\geq 2$ or $p=2$, $l=1$, $n$ even},\\
-\left(\frac{p-1}{2}\right)^{n-1}
\not\equiv 0
\ (\mathrm{mod}\ p),
&\mbox{if $p\neq 2$}.
\end{array}\right.
\end{equation}
\end{proposition}
\begin{proof}
By separating those terms in (\ref{eq:coeff_expl}) whose
denominators $j_1+\dots+j_m-k$ 
are multiples of $p^l$ from the others we obtain
\begin{equation}\label{eq:cAU1}
\widetilde{c}(q_1,\dots,q_m) = \frac{1}{p^l}A+U_0
\end{equation}
with
\begin{equation}\label{eq:cAU2}
A
=
\!\!\!\sum_{\substack{1\leq j_1\leq q_1,\\ \dots,\\1\leq j_m\leq q_m}}
\!\!\!\!\!
(-1)^{p^l\left\lfloor\frac{j_1+\ldots+j_m}{p^l}\right\rfloor}
\genfrac(){0pt}{}{\widetilde{m}}{k(j_1,\dots,j_m)}\frac{j_1!\cdots j_m!}{\left\lfloor\frac{j_1+\ldots+j_m}{p^l}\right\rfloor}
S(q_1,j_1)\cdots S(q_m,j_m)
\end{equation}
and $U_0\in\mathbb{Q}$, $v_p(U_0)>-l$.
Here we have used that  $j_1+\ldots+j_m-k$ is a multiple of $p^l$ 
if and only if 
$$k =k(j_1,\dots,j_m) = j_1+\ldots+j_m\ \mathrm{mod}\ p^l
=j_1+\ldots j_m -p^l\left\lfloor\frac{j_1+\ldots+j_m}{p^l}\right\rfloor.
 $$
From  Lemma~\ref{lemma:vp_fac_sum} and Remark~\ref{rem:vp_fac_sum} below it follows that
$v_p(j_1!\cdots j_m!)\geq v_p\left(\left\lfloor\frac{j_1+\ldots+j_m}{p^l}\right\rfloor\right)$,
where equality can only hold
if 
\begin{enumerate}
\item $j_i\leq p-1$, $i=1,\dots, m$, or 
\item $p\neq 2$, $l=1$, and there exists at least one index $i$, $2\leq i \leq m$ such that 
$j_i\in\{2,\dots,p-1,p+1,\dots,2p-1\}$, or 
\item $p=2$, $l=1$, $m=p^l=2$, and $(j_1,j_2)\in\{(1,3),(3,1)\}$, or
\item $p=2$, $l=1$, $m=p^l+1=3$, and $(j_1,j_2,j_3)\in\{(1,1,2), (1,2,1),(2,1,1)$, $(1,1,3), (1,3,1), (3,1,1)\}$.
\end{enumerate}
In the second case we have 
$S(q_1,j_1)\cdots S(q_m,j_m) =S(q_1,j_1)S(p^{e_2},j_2)\cdots S(p^{e_m},j_m)$
$\equiv 0 \ (\mathrm{mod}\ p)$,
which follows
from the fact that the Stirling numbers of the second kind satisfy
\begin{equation}\label{eq:stirling_power_mod}
S(p^e, j) \equiv \left\{\begin{array}{cl}
0 \ (\mathrm{mod}\ p), &\mbox{if $j\neq p^f$, $f=0,\dots,e$}, \\
1 \ (\mathrm{mod}\ p), &\mbox{if $j= p^f$, $f=0,\dots,e$}, 
\end{array}\right.\quad e\geq 0,\ j=1,\dots,p^e,
\end{equation}
see \cite[Theorem~4.2]{Howard}. In the first case this property 
of the Stirling numbers 
implies that
$S(q_1,j_1)\cdots S(q_m,j_m)=S(q_1,j_1)S(p^{e_2},j_2)\cdots S(p^{e_m},j_m)
\not\equiv 0 \ (\mathrm{mod}\ p)$ is possible only for
$j_2=\ldots=j_m=1$. 
We have thus established that
\begin{equation}\label{eq:vp_greater_0}
v_p\left(\frac{j_1!\cdots j_m!}{\left\lfloor\frac{j_1+\ldots+j_m}{p^l}\right\rfloor}
S(q_1,j_1)\cdots S(q_m,j_m)\right)\geq 0,
\end{equation}
where for $p\neq 2$ or $p=2$,  $l\geq 2$ equality is possible only for $j_1\leq p-1$ and $j_2=\ldots=j_m=1$.
In the latter case a simple calculation yields
$$
k(j_1,1,\dots, 1) = \left\{\begin{array}{cl}
j_1-1, &\mbox{if $m=p^l$},\\
j_1, &\mbox{if $m=p+1$},
\end{array}\right.
$$
and
$$
\left\lfloor\frac{j_1+\ldots+j_m}{p^l}\right\rfloor=1,$$
and we have $S(q_2,j_2)\cdots S(q_m,j_m)=1$.
Collecting in  (\ref{eq:cAU1}), (\ref{eq:cAU2}) the terms 
of $A$ with $j_1\leq \min(p-1, q_1)=p-1$,\footnote{$\min(p-1, q_1)=p-1$
because $q_1\geq p$ by assumption.} $j_2=\ldots=j_m=1$, and combining the other terms (which have $p$-adic valuation $>0$) divided by $p^l$ with $U_0$,
we  obtain (\ref{eq:cauv}) with
\begin{equation}\label{eq:a}
a = (-1)^p\sum_{j=1}^{p-1}\genfrac(){0pt}{}{\widetilde{m}}{k(j)}j!\,S(q_1,j),\quad
k(j)=\left\{\begin{array}{cl}
j-1,& \mbox{if $m=p^l$},\\
j,& \mbox{if $m=p^l+1$} 
\end{array}\right.
\end{equation}
and suitable $U\in\mathbb{Q}$ with $v_p(U)>-l$.

{\em Case $p\neq 2$.}
Because
$$
\widetilde{m} = \frac{p^l-1}{2}=\frac{p-1}{2}\big(1+p+\ldots+p^{l-1}\big)
$$
and $k(j)\leq p-1$,
Lucas's theorem, which states that
$$
\genfrac(){0pt}{}{\alpha_0+\alpha_1p+\ldots+\alpha_rp^r}{\beta_0+\beta_1p+\ldots+\beta_rp^r}\equiv
\genfrac(){0pt}{}{\alpha_0}{\beta_0}
\genfrac(){0pt}{}{\alpha_1}{\beta_1}\cdots
\genfrac(){0pt}{}{\alpha_r}{\beta_r}
\ \ (\mathrm{mod}\ p),\quad 0\leq\alpha_i,\beta_i\leq p-1
$$
(cf., e.g., \cite{Fine}),
implies
$$
\genfrac(){0pt}{}{\widetilde{m}}{k(j)}
\equiv\genfrac(){0pt}{}{\frac{p-1}{2}}{k(j)}\quad (\mathrm{mod}\ p).
$$
Substituting (\ref{eq:stirling}) in (\ref{eq:a}) we thus obtain
\begin{equation}\label{eq:a_as_sum}
a \equiv-\sum_{i}i^{q_1}(-1)^{i}\sum_{j}
\genfrac(){0pt}{}{\frac{p-1}{2}}{k(j)}
\genfrac(){0pt}{}{j}{i}(-1)^j\quad (\mathrm{mod}\ p).
\end{equation}
Using 
$$
\sum_{j=-\beta}^{\alpha}\genfrac(){0pt}{}{\alpha}{\beta+j}
\genfrac(){0pt}{}{\gamma+j}{\delta}(-1)^j
=(-1)^{\alpha+\beta}\genfrac(){0pt}{}{\gamma-\beta}{\delta-\alpha},\quad
\alpha, \beta, \gamma, \delta\in\mathbb{Z}, \ \alpha\geq 0, \ \gamma\geq\beta
$$
(cf., e.g., \cite[Eq. (5.24)]{GKP}), we obtain for the case $m=p^l$, 
\begin{align*}
\sum_{j=1}^{\frac{p+1}{2}}
\genfrac(){0pt}{}{\frac{p-1}{2}}{j-1}
\genfrac(){0pt}{}{j}{i}(-1)^j
&
=(-1)^{\frac{p+1}{2}}\genfrac(){0pt}{}{1}{i-\frac{p-1}{2}}\\
&=\left\{\begin{array}{cl}
(-1)^{\frac{p+1}{2}}, & \mbox{if $i=\frac{p-1}{2}$ or $i=\frac{p+1}{2}$},\\
0, & \mbox{otherwise},
\end{array}\right.
\end{align*}
and thus by substituting the sum over $j$ in (\ref{eq:a_as_sum})
\begin{align*}
a&\equiv \left(\frac{p-1}{2}\right)^{q_1}-
\left(\frac{p+1}{2}\right)^{q_1}
\equiv
 \left(\frac{p-1}{2}\right)^{n}-
\left(\frac{p+1}{2}\right)^{n}\\
&\equiv
\left(\frac{p-1}{2}\right)^{n}(1-(-1)^{n}) 
=\left\{\begin{array}{rl}
2\left(\frac{p-1}{2}\right)^{n}\ (\mathrm{mod}\ p), &\mbox{if $n$ even},\\
0\qquad\ (\mathrm{mod}\ p), &\mbox{if $n$ odd}.
\end{array}\right.
\end{align*}
Here we used $x^{q_1}\equiv x^n\ (\mathrm{mod}\ p)$ for $x\not\equiv 0 \ (\mathrm{mod}\ p)$, which follows
from $x^{n+1}
=x^{q_1}x^{p^{e_2}}\cdots x^{p^{e_m}}x
\equiv
x^{q_1}x^m=x^{q_1}x^{p^l}\equiv x^{q_1+1}\ (\mathrm{mod}\ p)$, which is
a consequence of Fermat's little theorem.
Similarly, for the case $m=p^l+1$,
\begin{equation*}
\sum_{j=1}^{\frac{p-1}{2}}
\genfrac(){0pt}{}{\frac{p-1}{2}}{j}
\genfrac(){0pt}{}{j}{i}(-1)^j
=(-1)^{\frac{p-1}{2}}\genfrac(){0pt}{}{0}{i-\frac{p-1}{2}} 
=\left\{\begin{array}{cl}
(-1)^{\frac{p-1}{2}}, & \mbox{if $i=\frac{p-1}{2}$},\\
0, & \mbox{otherwise},
\end{array}\right.
\end{equation*}
and thus
$$
a\equiv -\left(\frac{p-1}{2}\right)^{q_1}
\equiv -\left(\frac{p-1}{2}\right)^{n-1}\quad(\mathrm{mod}\ p).$$

{\em Case $p=2$, $l\geq 2$.} In this case the sum (\ref{eq:a}) reduces to 
the single term with $j=1$ where $S(q_1,j)=1$, and we obtain 
$$
a=\genfrac(){0pt}{}{\widetilde{m}}{k(j)}=
\genfrac(){0pt}{}{2^{l-1}-1}{0} = 1 \qquad \mbox{for $m=p^l$}
$$
and 
$$
a=\genfrac(){0pt}{}{\widetilde{m}}{k(j)}=
\genfrac(){0pt}{}{2^{l-1}}{1} = 2^{l-1}\equiv
0\quad (\mathrm{mod}\ 2)
\qquad\mbox{for $m=p^l+1$, $l\geq 2$}.
$$

{\em Case $p=2$, $l=1$, $m=p^l=2$.} In this case, equality
in (\ref{eq:vp_greater_0}) can possibly hold only for $(j_1,j_2)\in\{(1,1),(3,1)\}$, see the cases 1 and 3 above. Here we have already excluded  $(j_1, j_2)=(1,3)$, because $S(q_2,3)\equiv 0\ (\mathrm{mod}\ 2)$ if
$q_2=2^{e_2}\geq 3$ by (\ref{eq:stirling_power_mod}) (or, alternatively, by 
(\ref{eq:Sq3}) below).
We define $a$ in (\ref{eq:cauv}) as the sum of the terms in (\ref{eq:cAU2}) corresponding to 
$(j_1,j_2)\in\{(1,1),(3,1)\}$  and define $U\in\mathbb{Q}$ accordingly.
Noting $\widetilde{m}=1$, $k(3,1)= 0$, 
$\left\lfloor\frac{3+1}{2^1}\right\rfloor=2$, we
obtain $a\equiv S(q_1,1)+\frac{1}{2}3!S(q_1,3)
\equiv 1+q_1\equiv 1\ (\mathrm{mod}\ 2)$, if $q_1\geq 3$, and 
$a\equiv S(q_1,1)\equiv 1\ (\mathrm{mod}\ 2)$ otherwise. Here,
in the former case we have used 
\begin{equation}\label{eq:Sq3}
S(q,3)=\frac{3^{q-1}-1}{2}+1-2^{q-1}
=1+3+\ldots+3^{q-2}+1-2^{q-1}\equiv q \ (\mathrm{mod}\ 2)
\end{equation}
and the fact that $q_1$ is even, which for $n$ odd follows
from the assumption $q_1=n-1$, and for $n$ even from
$q_1=n-q_2=n-2^{e_2}$ where $e_2\geq 1$ if $n$ even.

{\em Case $p=2$, $l=1$, $m=p^l+1=3$.}
In this case, besides $q_2$ and $q_3$, which are powers of $2$ by construction, also $q_1$ is a power of $2$. In fact, $q_1,q_2,q_3$ are  pairwise distinct powers of $2$. It follows that $n=q_1+q_2+q_3$
is odd if and only if $q_i=1$ for a single $i\in\{1,2,3\}$.
In $v_p(j_1!\cdots j_m!)\geq v_p\left(\left\lfloor\frac{j_1+\ldots+j_m}{p^l}\right\rfloor\right)$ equality
is  possible only for 
$(j_1,j_2,j_3)\in\{(1,1,1), (1,1,2), (1,2,1),(2,1,1), (1,1,3), (1,3,1), (3,1,1)\}$. Because $S(q_i,1)=S(q_i,2)\equiv 1 \ (\mathrm{mod}\ 2)$
and $S(q_i,3)\equiv 0  \ (\mathrm{mod}\ 2)$ 
by (\ref{eq:stirling_power_mod}), it follows that  
in (\ref{eq:vp_greater_0}) equality holds precisely
if  $(j_1,j_2,j_3)\in\{(1,1,1)$, $(1,1,2), (1,2,1),(2,1,1)\}$.
For these $(j_1,j_2,j_3)$ we have $k=k(j_1,j_2,j_3)=1$ for $(j_1,j_2,j_3)=(1,1,1)$ and $k=0$ otherwise, and thus
$\genfrac(){0pt}{}{\widetilde{m}}{k}=
\genfrac(){0pt}{}{1}{k} = 1,$
so that
the corresponding terms in (\ref{eq:cAU2}) are all 
$\equiv 1\ (\mathrm{mod}\ 2)$. Such a term appears in (\ref{eq:cAU2})
if and only if $j_i\leq q_i$, $i=1,2,3$. Thus, the number of such terms
is $3$ if $q_i=1$ for one $i\in\{1,2,3\}$, i.e., if $n$ is odd, and $4$
otherwise. If we define $a$ as the sum of these $3$ respectively $4$ terms 
and define $U\in\mathbb{Q}$ accordingly, we obtain (\ref{eq:cauv}) with
$a\equiv 1 \ (\mathrm{mod}\ 2)$ if $n$ is odd and
$a\equiv 0 \ (\mathrm{mod}\ 2)$ if $n$ is even.
\end{proof}

\begin{lemma}\label{lemma:vp_fac_sum}  Let $m=p^l$ or $m=p^l+1$ 
for $p\geq 2$ prime and  $l\geq 1$, and $j_1,\dots, j_m\geq 1$.
Then
\begin{equation*}
v_p(j_1!\cdots j_m!)\geq v_p\left(\left\lfloor\frac{j_1+\ldots+j_m}{p^l}\right\rfloor\right),
\end{equation*}
where equality holds precisely if $$j_i\leq p-1, \ \ i=1,\dots,m,$$
or  
\begin{equation}\label{eq:case_equality_l1_mpl}
l=1,\ \ m=p, \ \ \mbox{and} \ \ \
j_k=2p-1, \  j_i=p-1,  \ i\neq k \ \ \mbox{for some $k=1,\dots,m$},
\end{equation}
 or  
\begin{align}\label{eq:case_euqality_l1}
&l=1,\  m=p+1, \
p\leq j_k\leq 2p-1,  \ 1 \leq j_i\leq p-1, \ i\neq k
\ \ \mbox{for some $k=1,\dots,m$}, \nonumber\\
&\mbox{and} \ \ p^2\leq j_1+\ldots+j_m\leq p^2+p-1.
\end{align}
\end{lemma}
\begin{proof}
{\em Case $m=p^l$.}
Let 
\begin{equation*}
r_i=\max\{s:\ p^s\leq j_i\}, \ \ i=1,\dots,m,  \quad \mbox{and}  \quad
r=\max\{r_1,\dots, r_m\}.
\end{equation*}
Then 
\begin{equation*}
v_p(j_1!\cdots j_m!) \geq r_1+\ldots+r_m\geq r.
\end{equation*}
On the other hand, $j_i\leq p^{r_i+1}-1\leq p^{r+1}-1$. Hence,
\begin{equation*}
\left\lfloor\frac{j_1+\ldots+j_m}{p^l}\right\rfloor\leq 
\left\lfloor\frac{m\cdot(p^{r+1}-1)}{p^l}\right\rfloor = p^{r+1}-1,
\end{equation*}
and thus 
\begin{equation*}
v_p\left(\left\lfloor
\frac{j_1+\ldots+j_m}{p^l}
\right\rfloor
\right)\leq r
\leq v_p(j_1!\cdots j_m!).
\end{equation*}
Now assume that equality holds. Then 
 $r_k=r$ for some $k$ and $r_i=0$
for $i\neq k$. Furthermore,
$v_p(j_k!)=r$ and thus $j_k\leq 2p-1$. Hence, $r=0$ or $r=1$.
In the case $r=0$ it follows $j_i\leq p-1$, $i=0,\dots,m$.
Assume $r=1$. Then
\begin{equation*}
v_p\left(\left\lfloor
\frac{j_1+\ldots+j_m}{p^l}
\right\rfloor
\right)=1, \ \ \mbox{thus} \ \
\frac{j_1+\ldots+j_m}{p^l}\geq\left\lfloor\frac{j_1+\ldots+j_m}{p^l}\right\rfloor\geq p
\end{equation*}
and
\begin{equation*}
p^{l+1}\leq j_1+\dots+j_m\leq 2p-1+(p^l-1)(p-1) = p^{l+1}-p^l+p
\end{equation*}
or $p^l\leq p$, which implies $l\leq 1$, and thus $l=1$ because
$l\geq 1$ by assumption.
For $l=1$ we have $j_1+\ldots+j_m=p^2$, which can only hold for
$j_k=2p-1$ and $j_i=p-1$, $i\neq k$. 

{\em Case $m=p^l+1$.} We first assume that 
$r_k=r$ and $r_i=0$, $i\neq k$ for some  $k=1,\dots,m$, where
$r_i$, $r$ are defined as before.
Also as before, this implies $r=0$ or $r=1$.
For the case $r=0$ we have
$j_i\leq p-1$, $i=0,\dots,m$. Thus,
\begin{equation*}
\left\lfloor\frac{j_1+\ldots+j_m}{p^l}\right\rfloor\leq 
\frac{(p^l+1)(p-1)}{p^l} < p
\end{equation*}
and
\begin{equation*}
v_p\left(\left\lfloor\frac{j_1+\ldots+j_m}{p^l}\right\rfloor\right)
= 0 = v_p(j_1!,\dots, j_m!).
\end{equation*}
Note that this is one of the cases for which equality holds.

For the case $r=1$ we have
$p\leq j_k\leq 2p-1$ and  $j_i\leq p-1$, $i\neq k$ for some $k=1,\dots,m$. Thus,
\begin{equation*}
\left\lfloor\frac{j_1+\ldots+j_m}{p^l}\right\rfloor\leq 
\frac{2p-1+p^l(p-1)}{p^l} \leq p+1
\end{equation*}
and
\begin{equation*}
v_p\left(\left\lfloor\frac{j_1+\ldots+j_m}{p^l}\right\rfloor\right)
\leq 1 = r = r_1+\ldots+ r_m = v_p(j_1!\cdots j_m!).
\end{equation*}
If equality holds, then 
\begin{equation*}
p^{l+1}\leq j_1+\dots+j_m \leq 2p-1+p^l(p-1)=p^{l+1}-p^l+2p-1
\end{equation*}
or $p^l\leq 2p-1$, which implies $l=1$ and $p^2\leq j_1+\ldots+j_m\leq p^2 +p-1$.

We now assume the opposite as before, namely  that $r_k=r$, $r_i=0$, $i\neq k$ for some $k=0,\dots,m$ does {\em not} hold.
Then 
$$v_p(j_1!\cdots j_m!)\geq r_1+\ldots+r_m \geq r+1.$$
On the other hand,
\begin{equation*}
\left\lfloor\frac{j_1+\ldots+j_m}{p^l}\right\rfloor\leq 
\frac{(p^l+1)(p^{r+1}-1)}{p^l} = 
p^{r+1}-1+p^{r+1-l}-p^{-l}\leq p^{r+2}-1.
\end{equation*}
Thus,
\begin{equation*}
v_p\left(\left\lfloor
\frac{j_1+\ldots+j_m}{p^l}
\right\rfloor
\right)\leq r+1
\leq v_p(j_1!\cdots j_m!).
\end{equation*}
If we assume that equality holds here, then $r_k=r$ for some $k$,
$r_h=1$ for some $h\neq k$, and $r_i=0$ for $i\neq k,h$. From
$v_p(j_h!)=1$ it follows $p\leq j_h\leq 2p-1$, and 
from
$v_p(j_k!)=r$ it follows $j_k\leq 2p-1$, and thus  $r\leq 1$, so $r=1$ 
because $r\geq r_h=1$.
Together with $j_h\leq 2p-1$ and $j_i\leq p-1$ for $i\neq k,h$,
and $v_p(\lfloor(j_1+\ldots+j_m)/p^l\rfloor)=2$,
this implies
\begin{equation*}
p^2\leq\left\lfloor\frac{j_1+\ldots+j_m}{p^l}\right\rfloor
\leq \frac{2(2p-1)+(p^l-1)(p-1)}{p^l}
= p-1+\frac{3p-1}{p^l},
\end{equation*}
and thus $p^{l+2}\leq p^{l+1}-p^l+3p-1$, which is impossible for 
$p\geq 2$ and $l\geq 1$.
\end{proof}

\begin{remark}\label{rem:vp_fac_sum}
If $p=2$, $l=1$, $m=p^l=2$, then 
(\ref{eq:case_equality_l1_mpl}) holds 
precisely for $(j_1,j_2)\in\{(1,3),(3,1)\}$.

If $p=2$, $l=1$, $m=p^l+1=3$, then  (\ref{eq:case_euqality_l1}) holds 
precisely for $(j_1,j_2,j_3)\in\{(1,1,2),(1,2,1),(2,1,1),(1,1,3),(1,3,1),(3,1,1)\}$. 

If $p\geq 3$, $l=1$, $m=p^l+1=p+1$, then (\ref{eq:case_euqality_l1}) can hold only if there exist
at least two indices $s\neq t$, such that $j_s,j_t\in\{2,\dots,p-1,p+1,\dots,2p-1\}$, because on the one hand
it cannot be the case that
 $j_i=1$ for all $i\neq k$ in (\ref{eq:case_euqality_l1}), which
 would  imply $j_1+\ldots+j_m=m-1+j_k\leq 3p-1<p^2$ contradicting
$j_1+\ldots+j_m\geq p^2$, and, on the other hand, if
 $2\leq j_h\leq p-1$ for some $h\neq k$ but
$j_i=1$ for $i\neq h,k$, then $j_k\geq p+1$, because $j_k=p$ would
imply $j_1+\ldots+j_m =  m-2 + j_h+j_k\leq 3p-2$ contradicting 
$j_1+\ldots+j_m \geq p^2$ again.
\end{remark}
\section{Numerical illustrations}\label{Sec:numerical_illustrations}

We verify the results of Propositions~\ref{prop:sum_m_2_mod_p} and \ref{prop:main_prop} 
by some explicit numerical calculations.
For given $n$, $p$ and suitable $q_1,\dots,q_m$ we compute 
the coefficient $c(q_1,\dots, q_m)$ using, e.g., Algorithm~\ref{alg:BCH}
from the appendix, and determine an integer $\hat{a}$,  which according to our theory  is expected to  be $\equiv a$ (or $\equiv -a$ in some cases) modulo $p$,
where the 
integer $a$ is   defined in these propositions. 
The details for the computation of $\hat{a}$ 
for given prime $p\geq 2$ and degree $n\geq p$
are  as follows.
\begin{itemize}
\setlength{\itemsep}{0.5pt}
  \setlength{\parskip}{0pt}
\item Determine $q_1,\dots, q_m$
\begin{itemize}
\setlength{\itemsep}{0.5pt}
  \setlength{\parskip}{0.5pt}
\item for $l=0$,  $n\geq p$,  $m=2$ according to 
 Algorithm~\ref{alg:witness}, lines 14--15;
\item for $l=1$, $m=2$ according to 
 Algorithm~\ref{alg:witness}, lines 18--26;
\item for $l\geq 1$, $m=p^l$ or $m=p^l+1$ according to 
Algorithm~\ref{alg:witness}, lines 33--34;\footnote{
If $l=1$, Algorithm~\ref{alg:witness} would branch to the simpler construction for the $q_i$  in lines 18--26. But the construction of lines 33--34 
works also for $l=1$.
Systematic computations that verify this construction for the case 
$l\geq 2$ are hardly feasible, as discussed in Section~\ref{SubSec:num_ill_l_gt_2} below. So we have to 
rely on the case
$l=1$ for such verifications.
}
 note that for $m=p^l+1$, $s_p(n)\neq p^l$ is required.
\end{itemize}
\item Compute $c=c(q_1,\dots,q_m)$ using, e.g.,  Algorithm~\ref{alg:BCH} from the appendix.
\item Set $\widetilde{c} = (-1)^nq_1!\cdots q_m!\cdot c$ and
determine integers $u$, $v$ and an exponent $e\geq 0$ such 
that $p\nmid u,v$ and $\widetilde{c}=\frac{u}{p^e v}$.
\item 
Set $\hat{a}=u\bar{v}\ \mathrm{mod}\ p$ and $\hat{U}=\widetilde{c}-\frac{\hat{a}}{p^e}$, 
where $\bar{v}$ is an inverse of $v$ modulo $p$.
\end{itemize}
 With an integer $y$ such that $v\bar{v}=1+yp$ 
 (note that $v\bar{v}\equiv 1 \ (\mathrm{mod}\ p)$)
 we have
$$\hat{U}=\widetilde{c}-\frac{\hat{a}}{p^e}=\frac{u}{p^ev}-\frac{u\bar{v}}{p^e}=\frac{uy}{p^{e-1}v},$$
 and thus $v_p(\hat{U})>-e$.
It follows that $\hat{a}$ is the  unique integer
$0\leq\hat{a}\leq p-1$ that satisfies
$$ \widetilde{c} = \frac{\hat{a}}{p^e}+\hat{U},\quad
 \mbox{for some} \ \hat{U}\in\mathbb{Q},\ v_p(\hat{U})>-e.$$

For $l\leq 1$, $m=2$
 we expect
$$
e=1\quad\mbox{and}\quad \hat{a} \equiv -(-1)^{q_2} a\equiv (-1)^{n+1}\ \ (\mathrm{mod}\ p)
$$
by Proposition~\ref{prop:sum_m_2_mod_p} and Remark~\ref{rem:sum_m_2_mod_p}.
Here we have used that $q_2$ is a multiple of $p-1$ by construction, and thus it is even for $p\neq 2$.

For $l\geq 1$, $m=p^l$ we expect
\begin{equation}\label{eq:expect_m}
e=l\quad\mbox{and}\quad\hat{a} \equiv a \equiv 2\left(\frac{p-1}{2}\right)^n\ \
(\mathrm{mod}\ p),\quad\mbox{if $p\neq 2$ and $n$ odd}
\end{equation}
by Proposition~\ref{prop:main_prop}. Note that in this case 
$c(q_1,\dots,q_m)=0$ if $p\neq 2$ and $n$ even, cf.~Remark~\ref{rem:1}.

For  $l\geq 1$, $m=p^l+1$ we similarly expect
\begin{equation}\label{eq:expect_m_plus_1}
e=l\quad\mbox{and}\quad\hat{a} \equiv a \equiv -\left(\frac{p-1}{2}\right)^{n-1}\ \
(\mathrm{mod}\ p),\quad\mbox{if $p\neq 2$.}
\end{equation}

\subsection{The case $l(n,p)\leq 1$} 
\begin{table}[!t]
{\small
\begin{center}
\addtolength{\tabcolsep}{-3.4pt}  
\begin{tabular}{cccccccc}
\hline
$n$ & $p$ & $l$ & $m$  & $(q_1,\dots,q_m)$ & $c(q_1,\dots,q_m)$ & $e$ & $\hat{a}$  \\
\hline
$26$&	$7$&	$1$&	$2$&	$(14, 12)$&	$-63102076049869/846912068365871834726400000$&	$1$&	$6$\\
&	$7$&	$1$&	$7$&	$(14, 7, 1, 1, 1, 1, 1)$&	$0$&	$0$&	$0$\\
&	$7$&	$1$&	$8$&	$(7, 7, 7, 1, 1, 1, 1, 1)$&	$5260127/12693891496366080000$&	$1$&	$4$\\
\hline
$27$&	$7$&	$1$&	$2$&	$(21, 6)$&	$-6333157/33967061565476143104000$&     $1$&     $1$\\
&	$7$&	$1$&	$7$&	$(21, 1, 1, 1, 1, 1, 1)$&	$-1970755117/6416000517923271475200000$&	$1$&	$5$\\
&	$7$&	$1$&	$8$&	$(14, 7, 1, 1, 1, 1, 1, 1)$&	$2609686559/51142033113881149440000$&	$1$&	$5$\\
\hline
$28$&	$7$&	$0$&	$2$&	$(22, 6)$&	$252293307089/10162944820390462016716800000$&	$1$&	$6$\\
\hline
\end{tabular}
\addtolength{\tabcolsep}{3.4pt}
\caption{Results of various computations for the case $l=l(n,p)\leq 1$.}\label{tbl:l1}
\end{center}
}
\end{table}

Some results of such computations for $l=(n,p)\leq 1$ with $p=7$ can 
be found in Table~\ref{tbl:l1}. For one of its
 entries we  give  the details of the computations below.
It is readily verified that  all results for $e$ and $\hat{a}$ 
in the table are as expected.
 
\medskip 
 
\noindent{\bf Example:} 
We consider 
$n=27=3\cdot 7+6$, $p=7$, $l=l(n,p)=\lfloor\log_7(3+6)\rfloor=1$, $m=p^l+1=8$.
\begin{itemize}
\setlength{\itemsep}{0.5pt}
  \setlength{\parskip}{0.0pt}
\item
$n=27=\underbrace{7+7}_{q_1}
+\underbrace{7}_{q_2}
+\underbrace{1}_{q_3}
+\underbrace{1}_{q_4}
+\underbrace{1}_{q_5}
+\underbrace{1}_{q_6}
+\underbrace{1}_{q_7}
+\underbrace{1}_{q_8}$\\
such that
$(q_1,\dots,q_m)=(14,7,1,1,1,1,1,1).$
\item
$c=c(14,7,1,1,1,1,1,1)=2609686559/51142033113881149440000.$
\item
$\widetilde{c}=-14!\,7!\cdot c= -2609686559/116396280 = \frac{u}{p^e v}$
with
$e=1$, $u=-2609686559\equiv 3\ (\mathrm{mod}\ 7)$, $v=16628040\equiv 2\ (\mathrm{mod}\ 7)$.
\item
$\bar{v}\equiv 4\ (\mathrm{mod}\ 7)$; thus, $\hat{a}= u\bar{v}\ \mathrm{mod}\ 7 = 5$, $\hat{U}=\widetilde{c}-\frac{5}{7}=-384689537/16628040$. 
\end{itemize} 
We verify  (\ref{eq:expect_m_plus_1}),
$$a\equiv-\left(\frac{p-1}{2}\right)^{n-1}=-3^{26}=-(3^6)^4\cdot 3^2\equiv -9 \equiv 5
\equiv \hat{a} \ \ (\mathrm{mod}\ 7),$$ where we have used 
$3^6=3^{7-1}\equiv 1 \ (\mathrm{mod}\ 7)$ by
Fermat's little theorem.

\subsection{The case $l(n,p)\geq 2$}\label{SubSec:num_ill_l_gt_2}
It can be shown
that the smallest degree $n$ 
 for which
$l(n,p)$ is equal to a  given number $l\geq 2$ is given by
$
n=2p^x-1$,  where $x=\frac{p^l-1}{p-1}$.
Some values of these degrees are shown in the following table.
\begin{center}
$\min\{n:\ l(n,p)= l\}$:\qquad
\begin{tabular}{c|ccc}
$p$ & $l=2$ & $l=3$ & $l=4$ \\
\hline
$2$ & $15$ & $255$ & $65535$ \\
$3$ & $161$ & $3188545$ & $\approx 2.43\cdot 10^{19}$ \\
$5$ & $31249$ & $\approx 9.31\cdot 10^{21}$ & $\approx 2.19\cdot 10^{109}$
\end{tabular}
\end{center}
Explicit computations with 
such degrees are obviously 
impossible in most cases. 
Some results of feasible computations can be found in Table 2.
\begin{table}[!t]
{\small
\begin{center}
\addtolength{\tabcolsep}{-3.4pt}  
\begin{tabular}{cccccccc}
\hline
$n$ & $p$ & $l$ & $m$  & $(q_1,\dots,q_m)$ & $c(q_1,\dots,q_m)$ & $e$ & $\hat{a}$  \\
\hline
$161$&	$3$&	$2$&	$9$&	$(81, 27, 27, 9, 9, 3, 3, 1, 1)$&	$\mbox{(168-digits number})/(\mbox{248-digits number})$&	$2$&	$2$\\
$242$&	$3$&	$2$&	$10$&	$(81, 81, 27, 27, 9, 9, 3, 3, 1, 1)$&	$\mbox{(288-digits number})/(\mbox{408-digits number})$&	$2$&	$2$\\
$255$&	$2$&	$3$&	$8$&	$(128, 64, 32, 16, 8, 4, 2, 1)$&	$\mbox{(330-digits number})/(\mbox{460-digits number})$&	$3$&	$1$\\
\hline
\end{tabular}
\addtolength{\tabcolsep}{3.4pt}
\caption{Results of various computations for the case $l=l(n,p)\geq 2$.
Here $n=242$ is the smallest {\em even}
degree that satisfies $l(n, 3)=2$.
The results 
conform with (\ref{eq:expect_m}), (\ref{eq:expect_m_plus_1}), and (\ref{eq:cauv1}), case $p=2$, respectively.
}\label{tbl:l23}
\end{center}
}
\end{table}

\subsection{A simpler construction for the case $l(n,p)\geq 2$ ?}
If $l(n,p)=1$, then there exists a simpler method 
than the one of Section~\ref{SubSec:l_gt_1} for 
 obtaining 
a partition $(q_1,\dots, q_m)$ of $n$ ($m\geq 1$, $q_i\geq 1$, $q_1+\ldots +q_m=n$) that satisfies (\ref{eq:vp_den_c}), namely the method of 
Section~\ref{SubSec:l_leq_1} which produces such a partition 
of the form $(n-k, n)$. 
If on the other hand $l(n,p)\geq 2$, then no such partition of length $m=2$
can exist. This follows from the explicit formula (\ref{eq:m2_expl})
involving Bernoulli numbers  and the fact that the denominators of the
Bernoulli numbers are square-free, which is a consequence of the
von Staudt–Clausen theorem,~cf.~\cite{Carlitz}.

We define the sets
\begin{align*}
Q(n,p) =\big\{(q_1, \dots, q_m):\ &m\geq 1,\ q_1\geq q_2\geq\ldots \geq q_m\geq 1,
\ q_1+\ldots+q_m=n, \\
&v_p(\denom(c(q_1,\dots,q_m)))=v_p(n!)+l(n,p)\big\}
\end{align*}
consisting of all partitions $(q_1, \dots, q_m)$ of $n$ in descending order 
that satisfy (\ref{eq:vp_den_c}).
Note that because of the invariance of $c(q_1,\dots,q_m)$ under
permutations of the $q_i$ (cf.~Remark~\ref{rem:1}),  all possible values for 
the coefficients of degree $n$  already occur under these special
ones corresponding to partitions of $n$ in descending order.

In particular, for $p=2$ and degrees $n=15,23,27,29,30,31$  
that satisfy $l(n,2)=2$, we obtain\footnote{
Without much sophistication, 
we compute the sets $Q(n,p)$
by an exhaustive search under all possible partitions. For example, for
$n=31$, there are 6842 partitions in descending order that have to be 
examined.
For $n=161$, which is the smallest degree $n$ satisfying $l(n,3)=2$, the number
of such partitions is 118159068427, a number far to large 
for a exhaustive search to be feasible.
 Therefore, we limit ourselves to the case $p=2$.
}
\begin{align*}
&Q(15, 2) = \{(8, 4, 2, 1)\},\ \
Q(23, 2) = \{(16, 4, 2, 1)\},\ \
Q(27, 2) = \{(16, 8, 2, 1)\},\\
&Q(29, 2) = \{(16, 8, 4, 1)\},\ \
Q(30, 2) = \{(16, 8, 4, 2)\},\ \
\end{align*}
and
\begin{align*}
Q(31, 2) = \big\{&(24, 4, 2, 1),\
(20, 8, 2, 1),\
(18, 8, 4, 1),\
(17, 8, 4, 2),\
(16, 12, 2, 1),\\
&(16, 10, 4, 1),\
(16, 9, 4, 2),\
(16, 8, 6, 1),\
(16, 8, 5, 2),\
(16, 8, 4, 3)\big\}.
\end{align*}
Here for each  $n\in\{15,23,27,29,30\}$ the set $Q(n,2)$ consists of a single partition of $n$, which therefore must be the one defined by Algorithm~\ref{alg:witness}, lines 33--34 (which is in descending order
by construction).
Permuting the components of the partitions in $Q(31,2)$ in such a way that the powers of 2
appear from position 2 in descending order we obtain
\begin{align*}
\widetilde{Q}(31,2) =\big\{&
(24,  4, 2, 1),\
(20,  8, 2, 1),\
(18,  8, 4, 1),\
(17,  8, 4, 2),\
(12, 16, 2, 1),\\
&(10, 16, 4, 1),\
(9,  16, 4, 2),\
(6,  16, 8, 1),\
(5,  16, 8, 2),\
(3,  16, 8, 4)\big\}.
\end{align*}
Now each of these partitions $(q_1,\dots, q_m)$ is of the form 
described in Section~\ref{SubSec:l_gt_1} such that Proposition~\ref{prop:main_prop} applies to them.

Thus, at least in the few cases just discussed, the construction of 
the $q_1,\dots,q_m$ 
in Section~\ref{SubSec:l_gt_1} is essentially the only possible one.
In any case, this suggests that for $l(n,p)\geq 2$, this construction cannot be significantly simplified.

\appendix
\section{An efficient algorithm for the computation of  BCH coefficients}
\begin{algorithm}\label{alg:BCH} 
\caption{Efficient computations of BCH coefficients}
\DontPrintSemicolon
\KwIn{$(q_1,\dots,q_m)\in\mathbb{N}^m_{>0}$, $\mathit{Afirst}\in\{\mathrm{true}, \mathrm{false}\}$}
\KwOut{$\coeff(w,\log(\ee^{\AA}\ee^{\BB}))$ for $w=\AA^{q_1}\BB^{q_2}\cdots$  or $w=\BB^{q_1}\AA^{q_2}\cdots$ depending on $\mathit{Afirst}$ }
$N:=\sum_{i=1}^mq_i$\;
$d:=N!\cdot d_N$\;
$C:=(0)\in\mathbb{Z}^{N\times N}$\;
$\mathit{Acurrent}:=\mathit{Afirst}$\;
\If{$m$ {\rm is even}}{$\mathit{Acurrent}:={\rm\bf not}\,  \mathit{Afirst}$}
$n:=0$\;
\For{$i:=m,m-1,\dots,1$}
{ \label{line:9}
	\For{$r:=1,\dots,q_i$}
	{
		$n:=n+1$\;
		$h:=0$\;
		\uIf{$i=m$}
		{
			$h:=d/n!$\;
		}
		\ElseIf{$\mathit{Acurrent}$ {\rm\bf and} $i=m-1$}
		{
			$h:=d/(r!q_{m}!)$\;
		}
		$C_{1,n}:=h$\;
		\For{$k:=2,\dots,n-1$}
		{
			$h:=0$\;
			\For{$j:=1,\dots,r$}
			{
				\If{$n>j$ {\rm\bf and} $C_{k-1,n-j}\neq 0$}
				{
					$h:=h+C_{k-1,n-j}/j!$\;
				}
			}
			\If{$\mathit{Acurrent}$ {\rm\bf and} $i\leq m-1$}
			{
				\For{$j:=1,\dots,q_{i+1}$}
				{
					\If{$n>r+j$ {\rm\bf and} $C_{k-1, n-r-j}\neq 0$}
					{
						$h:=h+C_{k-1,n-r-j}/(r!j!)$\;
					}
				}				
			}
			$C_{k,n}:=h$\;
		}
		$C_{n,n} := d$\;		
	}	
	$\mathit{Acurrent}:={\rm\bf not}\, \mathit{Acurrent}$\;
}
{\bf return} $\frac{1}{d}\sum_{k=1}^N(-1)^{k+1}C_{k,N}/k$\;
\end{algorithm}
In Algorithm~\ref{alg:BCH} we present a  new 
  method
for the efficient computation of BCH coefficients, using a
self-explanatory
pseudocode, which can  straightforwardly
be implemented in any
general purpose programming language\footnote{Depending on the available integer data type,
the  size of the degrees $N$ 
may be limited. 
For standard 64-bit integers, $N\leq 19$, and for 128-bit integers
(which, e.g., are available as numbers of type \texttt{\_\_int128\_t}
for many compilers for the C programming language on modern computer architectures), $N\leq 30$.
Higher degrees usually require a library for multi-precision integer arithmetic.} or 
any computer algebra system.
%
%
An implementation in the Julia programming language is available at \cite{HHjulia}.
The following comments 
should provide sufficient evidence for the correctness of the algorithm. 
\begin{description}
\setlength{\itemsep}{2pt}
  \setlength{\parskip}{2pt}
\item[\it Input:] 
We consider the  word 
$w=\AA^{q_1}\BB^{q_2}\cdots (\AA\lor\BB)^{q_m}$ 
or
$w=\BB^{q_1}\AA^{q_2}\cdots (\BB\lor\AA)^{q_m}$ 
as a concatenation of $m$ alternating  blocks of  $\AA$s or $\BB$s whose lengths are 
$q_1,\dots,q_m$.
The boolean variable $\mathit{Afirst}$ indicates whether the first block
is an $\AA$-block (or otherwise a  $\BB$-block).
\item[\it Line 1:]$N=q_1+\ldots+q_m$ is the length of the word $w$.
\item[\it Line 2:] $ d=n!\,d_N$ is the common denominator for all coefficients of degree $\leq N$ defined by (\ref{eq:d_n}).
\item[\it Line 3:] The array $(C_{k,n})\in\mathbb{Z}^{N\times N}$ 
is initialized to zero. It will eventually contain 
 $C_{k,n} = d\cdot\coeff(v(n), Y^k)$, $k=1,\dots,n$, $n=1,\dots,N$,
where 
\begin{equation*}
Y=\ee^{\AA}\ee^{\BB}-1=\sum_{i+j>0}\frac{1}{i!j!}\AA^i\BB^j,
\end{equation*}
and $v(n)=w_{N-n+1}\cdots w_N$ is the right subword of $w=w_1\cdots w_N$ 
of length $n$
starting
at position $N-n+1$.

\item[\it Lines 9--38:] 
The outermost loop over $i$ processes the $m$ blocks in reverse order.
The boolean variable $\mathit{Acurrent}$ indicates whether the  current $i$-th block
is an $\AA$-block.

\item[\it Lines 10--36:]The loop over $r=1,\dots,q_{i}$ combines
with the outer loop over $i$ to form a loop over 
$n=r+q_{i+1}+\ldots+q_m$ which processes the right subwords $v(n)$
of lengths $n$.

\item[\it Lines 12--18:]If $k=1$ then the current right subword $v(n)$ can only contribute 
to $C_{k,n}=C_{1,n}=d\cdot\coeff(v(n),Y)$,
if it has the form $v(n)=\AA^s\BB^t$ with $s+t=n$, and thus if it is contained in the last two blocks.
%
This contribution is $d/n!$ if
 $v(n)$ is entirely contained  in the last (i.e., the $m$-th) block
 such that $v(n)=\AA^n$ or $v(n)=\BB^n$,
 or it
is $d/(r!q_m!)$ if $v(n)$ is contained in the last two blocks,
where the next to last (i.e., the $(m-1)$-th) block has to be an
$\AA$-block such that $v(n)=\AA^r\BB^{q_m}$.

\item[\it Lines 19--34:]Let $u(n,j)$ denote the left subword of $v(n)$ of
length $j$ such that
$$v(n)=u(n,j)v(n-j),\quad j=0,\dots,n.$$
For $k=2,\dots, n-1$ we have
$$
\coeff(v(n),Y\cdot Y^{k-1})
=\sum_{j=0}^n\coeff(u(n,j),Y)\cdot\coeff(v(n-j),Y^{k-1}).
$$
Here we have $\coeff(v(n-j),Y^{k-1})=0$  for $j=n$.
Similarly as before (cf.~lines 12--18), 
we have
$\coeff(u(n,j),Y)\neq 0$ only if $j\geq 1$ and if either $u(n,j)$ is 
entirely contained in the current $i$-th block
(or, more precisely, the current right subblock of length $r$ of the $i$-th block),
or if it is entirely 
contained in the  union of the $i$-th and the $(i+1)$-th block, where
the   $i$-th block has to be an $\AA$-block.
In the former case $u(n,j)=\AA^j$ or $u(n,j)=\BB^j$, $j=1,\dots, r$ such that 
$\coeff(u(n,j),Y)=1/j!$, and in the latter case $u(n,j)=
\AA^r\BB^{j_1}$ with $r+j_1=j$, $j_1=1,\dots,q_{i+1}$
such that $\coeff(u(n,j),Y)=1/(r!j_1!)$.
It follows 
\begin{align*}
C_{k-1,n} &= d\cdot \coeff(v(n),Y^k)
= \sum_{j=1}^r\frac{1}{j!}C_{k-1,n-j}+
   f_i\sum_{j_1=1}^{q_{i+1}}\frac{1}{r!j_1!}C_{k-1,n-r-j_1} ,
\end{align*}
where $f_i=1$ if the $i$-th block is an $\AA$-block and 
$f_i=0$ otherwise. This sum is computed in lines 20--33.
Note that here $C_{k-1,n-j}$ and $C_{k-1,n-r-j_1}$ either are understood to
be $=0$ if the second index is 0,
or 
they have already been computed
during a previous pass of the loop over $n$ (i.e., the loops over
$i$ and $r$ combined).

Obviously the tests for $C_{k-1,n-j}\neq 0$ respectively
$C_{k-1,n-r-j}\neq 0$ in lines 22 and 28 are  not strictly necessary, 
but are there for efficiency reasons.

\item[\it Line 35:] 
A word of degree $n$ occurs in
$Y^n$ if and only if for all $i=1,\dots, n$,
its $i$-th letter corresponds to exactly one term of degree 1 
of the $i$-th factor $Y$ of $Y^N$. Thus, $v(N)$ occurs in $Y^n$ exactly
once and with coefficient $1$ so that $C_{n,n} = d\cdot\coeff(v(n),Y^n)=d$.

Note that the case $k=n$ could also be handled by the above loop  over $k$. 
Here it is handled separately for efficiency and because it is so simple.

\item[\it Line 39:] The final result is computed 
according to 
$$\coeff(w,\log(\ee^{\AA}\ee^{\BB})) = 
\sum_{k=1}^N\frac{(-1)^{k+1}}{k}\coeff(w, Y^k) 
=\frac{1}{d}\sum_{k=1}^N\frac{(-1)^{k+1}}{k}C_{k,N}.
$$
\end{description}

The main feature of the algorithm is that 
it performs all of its computations in integer arithmetic. This
means in particular, that the divisions
in lines 14, 16, 23, 29, and the divisions by $k$ in line 39
never have a remainder. (Of course, this does not apply to the
final division by $d$ in line 39.) To prove this, it is not enough
to know that the final result is a rational number with a denominator that
is divisible by $d=N!d_N$. It must also be ensured that during  the calculation  no intermediate results not representable as integers can occur, 
which cancel out at the end. 
Without going into details, this holds because
the computations of the algorithm follow  the 
same pattern as
the calculation of the common denominator  $d=D_N=N!d_N$
 in the proof of  \cite[Proposition~1]{HHdenom},
 where the generic case is assumed and no cancellations are  taken into account.


\end{document}